\documentclass[a4paper,12pt]{amsart}
\usepackage{enumerate}
\usepackage{amsmath,amssymb,latexsym,amsxtra,amscd,amsthm}
\usepackage[mathscr]{eucal}
\usepackage{latexsym}
\usepackage{amsfonts}
\usepackage{dsfont}
\usepackage{enumerate}
\usepackage{pstricks}
\usepackage{pst-node}
\usepackage{pst-plot}
\usepackage{pst-all}
\numberwithin{equation}{section}

\newtheorem{theorem}{Theorem}[section]
\newtheorem{lemma}[theorem]{Lemma}
\newtheorem{definition}[theorem]{Definition}
\newtheorem{example}[theorem]{Example}
\newtheorem{remark}[theorem]{Remark}

\newtheorem{proposition}[theorem]{Proposition}

\newcommand{\hull}[1]{\text{\rm{conv}}(#1)}
\newcommand{\lh}{\mathcal L(H)}
\newcommand{\set}[2]{\ensuremath{\{ #1: #2\}}}

\begin{document}

\title{Equilateral weights on the unit ball of $\mathds R^n$}
\author{Emmanuel Chetcuti \and Joseph Muscat}
\address{E. Chetcuti\\
Department of Mathematics\\
Faculty of Science\\
University of Malta\\
Msida MSD 2080, Malta}
\email{emanuel.chetcuti@um.edu.mt}
\address{J. Muscat\\
Department of Mathematics\\
Faculty of Science\\
University of Malta\\
Msida MSD 2080, Malta}
\email{joseph.muscat@um.edu.mt}
\date{\today}
\maketitle

\begin{abstract}
An equilateral set (or regular simplex) in a metric space $X$, is a set $A$ such that the distance between any pair of distinct members of $A$ is a constant.  An equilateral set is standard if the distance between distinct members is equal to $1$.  Motivated by the notion of frame-functions, as introduced and characterized by Gleason in \cite{Gl}, we define an equilateral weight on a metric space $X$ to be a function $f:X\longrightarrow \mathds R$ such that $\sum_{i\in I}f(x_i)=W$,  for every maximal standard equilateral set $\{x_i:i\in I\}$ in $X$, where $W\in\mathds R$ is the weight of $f$.  In this paper we  characterize the equilateral weights associated with the unit ball $B^n$  of $\mathds R^n$ as follows: For $n\ge 2$, every equilateral weight on $B^n$ is  constant.
\end{abstract}

\section{Introduction}
Equilateral sets have been extensively studied in the literature for a number of metric spaces \cite{Bl}.  An equilateral set (or regular simplex) in a metric space $X$, is a set $A$ so that the distance between any pair of distinct members of $A$ is $\rho$, where $\rho\neq 0$ is a constant. The equilateral dimension of $X$  is defined to be $\sup\set{|A|}{ A \text{ is an equilateral set in }X}$.

Suppose that $\{x_1,\dots,x_k\}$ is an equilateral set in $\mathds R^n$  (equipped with the $\ell_2$-norm).  Then the vectors $v_i:=x_{i+1}-x_1$ for $i=1,\dots,k-1$ are linearly independent.  Indeed,  let $A$ be the $(k-1)\times (k-1)$ matrix $(a_{ij})$ defined by $a_{ij}:=\langle v_i,v_j\rangle$.  Then $a_{ij}=\frac{\rho^2}{2}(1+\delta_{ij})$ where $\rho\neq 0$ is a constant and $\delta_{ij}$ is the Kronecker delta.  Let $\{e_1,\dots,e_{n}\}$ be the canonical basis of $\mathds R^{n}$ and let $B$ be the $n\times(k-1)$ matrix $(b_{ij})$ defined by $b_{ij}:=\langle v_j,e_i\rangle$.  Since $A=B^\ast B$ and $A$ is clearly non-singular, we deduce that $B$ is non-singular, i.e. the vectors  $v_i:=x_{i+1}-x_1$ for $i=1,\dots,k-1$ are linearly independent and therefore $k\le n+1$.  To see that the equilateral dimension of $\mathds R^n$  (equipped with the $\ell_2$-norm) is $n+1$ observe that the set   $\{x_1-c,\dots,x_k-c\}$ where  $c:=\frac{1}{k}\sum_{i=1}^k x_i$  has linear dimension $k-1$ and so if $k<n+1$,  there exists a unit vector $u\in\mathds R^n$ such that $u\,\bot\, x_i-c$ for each $i=1,\dots,k$, and therefore the set $\{x_1,\dots,x_k\}$ can be enlarged to a bigger equilateral set in $\mathds R^n$. 
Let us only mention here  that the situation is far more complicated for the other $\ell_p$-norms \cite{Pe,KoLaSc,AlPu} (and others).

An equilateral set in $\mathds R^n$ is \emph{standard} if the distance between distinct points is equal to $1$.  If $\{x_1,\dots,x_k\}$ is a standard equilateral set in $\mathds R^n$, its centre  $\frac{1}{k}\sum_{i=1}^k x_i$ will be denoted by $c(x_1,\dots,x_k)$.  The \emph{radius} of $\{x_1,\dots,x_k\}$ is $\bigl\Vert x_i-c(x_1,\dots,x_k)\bigr\Vert$ and is denoted by $\beta_k$.  A simple calculation yields 
\begin{align*}
\beta_k=\Bigl\Vert x_i-c(x_1,\dots,x_k)\Bigr\Vert=&\frac{1}{k}\Bigl\Vert \sum_{\substack{1\le j\le k\\j\neq i}}(x_j-x_i)\Bigr\Vert\\
=&\frac{1}{k}\sqrt{k-1+\frac{(k-1)(k-2)}{2}}=\sqrt{\frac{k-1}{2k}}.
\end{align*}
If $x_{k+1}$ is another point in $\mathds R^n$ such that $\{x_1,\dots,x_k,x_{k+1}\}$ is again a standard equilateral set, then $x_{k+1}-c(x_1,\dots,x_k)$ is orthogonal to $x_{i}-c(x_1,\dots,x_k)$ for every $i=1,\dots,k$, and thus
\[\Bigl\Vert x_{k+1}-c(x_1,\dots,x_k)\Bigr\Vert=\sqrt{1-\beta_k^2}=\sqrt{\frac{k+1}{2k}}.\]
We will call $\alpha_{k+1}:=\sqrt{\frac{k+1}{2k}}$ the \emph{perpendicular height} of $\{x_1,\dots,x_k,x_{k+1}\}$.

We shall now introduce the notion of equilateral weights.  The motivation behind this definition is the notion of frame functions.  These were introduced and characterized by Gleason \cite{Gl} in his famous theorem describing the measures on the closed subspaces of a Hilbert space.   Gleason's Theorem is of utmost importance in the laying down of the foundations of quantum mechanics \cite{Varadarajan,Pt-Pu,Gudderbook,Dv,Ha} (and others).   Let $S(0,1)$ denote the unit sphere of a Hilbert space $H$.  A function $f:S(0,1)\to\mathds R$ is called a frame function on $H$ if there is a number $w(f)$, called the weight of $f$, such that $\sum_{i\in I}f(u_i)=w(f)$ for every orthonormal basis $\{u_i:i\in I\}$ of $H$.  We recall that a bounded operator $T$ on $H$ is of trace-class if the series $\sum_{i\in I}\langle Tu_i,u_i\rangle$ converges absolutely for any orthonormal  basis $\set{u_i}{i\in I}$ of $H$.  (It is well-known that if the series converges for an orthonormal basis $\set{u_i}{i\in I}$ then it converges for any orthonormal basis and the sum does not depend on the choice of the basis.)  Clearly, if $T$ is self-adjoint and of trace-class the function $f_T(x)=\langle Tx,x\rangle$  $(x\in S(0,1))$ defines a continuous frame function on $H$.  Gleason's Theorem says that  when $\dim H\ge 3$ every bounded frame function arises in this way.  The heart of the proof of Gleason's Theorem is the treatment of the case when $H$ is the real three-dimensional Hilbert space $\mathds R^3$.  In fact all the other cases can be reduced to this case.  Thus, as a matter of fact, it can be said that the crux of this theorem can be rendered to the following statement:
\emph{
For every bounded frame function $f$ on $\mathds R^3$ there exists a symmetric matrix $T$ on $\mathds R^3$ such that $f(u)=\langle Tu,u\rangle$ for every unit vector $u\in\mathds R^3$. } The notion of frame functions and the fact that an orthonormal basis of $\mathds R^3$ is simply a maximal  equilateral set on the unit sphere of $\mathds R^3$, suggest the following definition:

\begin{definition}
Let $X$ be a metric space and let $W\in\mathds R$.  An \emph{equilateral weight on $X$ with weight $W$} is a function $f:X\longrightarrow \mathds R$ such that
\[\sum_{i\in I}f(x_i)=W\]
whenever $\set{x_i}{i\in I}$ is a maximal standard equilateral set in $X$.
\end{definition}

Given a metric space, can one describe the equilateral weights associated with it?

\begin{example} Every equilateral weight on $\mathds R^2$ is constant.  First observe that for every pair of points $x$ and $y$ in $\mathds R^2$ there are points $x_1$, $x_2$, $\dots$, $x_n$ in $\mathds R^2$ such that $\Vert x_1-x\Vert=\Vert x_{i+1}-x_i\Vert=\Vert y-x_n\Vert =1$ for every $i=1,\dots,n-1$.  Thus, it suffices to to show that $f(x)=f(y)$ for all $x,y\in \mathds R^2$ satisfying $\Vert x-y\Vert=1$.  Let $x,y\in\mathds R^2$ such that $\Vert x-y\Vert=1$.  Observe that if $\{a,b,c\}$ and $\{d,b,c\}$ are the vertices of two unit equilateral triangles and $f$ is an equilateral weight, then $f(a)=f(d)$.  Thus,  $f$ takes the constant value $f(x)$ on the circle with centre $x$ and radius $\sqrt{3}$, and the constant value $f(y)$ on the circle with centre $y$ and radius $\sqrt{3}$.  Since these circles intersect, it follows that $f(x)=f(y)$.  Using a similar argument but replacing $\sqrt{3}$ with $2\alpha_{n+1}$, one can easily show that every equilateral weight on $\mathds R^n$ is constant.  The same cannot be said for $\mathds R$ -- it is easy to find non-trivial equilateral weights on $\mathds R$.
\end{example}

\begin{example}\label{Gleason} Let $S$ be the sphere in a Hilbert space $H$ with centre $0$ and radius $1/\sqrt 2$.  Two vectors $u$ and $v$ in $S$ satisfy $\Vert u-v\Vert=1$ if, and only if, $\langle u,v\rangle=0$. Thus, each maximal standard equilateral set in $S$ corresponds to a rescaling of some orthonormal basis of $H$ by a factor of $1/\sqrt 2$.  It is clear therefore that the equilateral weights on $S$ correspond to the frame-functions on $H$ (composite with a rescaling by a factor of $\sqrt 2$).  Thus, in view of  Gleason's Theorem  if $\dim H\ge 3$ and $f$ is a bounded equilateral weight on $S$, there exists a self-adjoint, trace-class operator $T$ such that
\[f(u)=\langle Tu,u\rangle\]
for all $u\in S$. Let us emphasize that such a description does not hold when $\dim H=2$ and that the assumption of boundedness is not redundant when $\dim H$ is finite.  It known that $\mathds R^n$ admits frame functions that are unbounded and that therefore cannot be described by such an equation (see \cite[Proposition 3.2.4]{Dv}).
\end{example}

By contrast, the boundedness assumption is superfluous when the space is infinite dimensional.  This surprising result is due to Dorofeev and Sherstnev \cite{Dorofeev-Sherstnev} and allows us to describe the equilateral weights associated with the metric space $S$ of an infinite dimensional Hilbert space directly from Gleason's Theorem.

\begin{proposition}
Let $H$ be an infinite dimensional Hilbert space and let $S$ be the sphere in $H$ with centre $0$ and radius $1/\sqrt 2$.  If $f$ is an equilateral weight on $S$, then there exists a self-adjoint, trace-class operator $T$ on $H$ such that $f(u)=\langle Tu,u\rangle$ for every vector $u$ in $S$.
\end{proposition}

The aim of the present paper is to describe the equilateral weights associated with another bounded metric space; namely the unit ball of $\mathds R^n$.

\section{Standard equilateral sets in the unit ball of $\mathds R^n$}

In what follows  we will be interested in standard equilateral sets contained in the (closed) unit ball of $\mathds R^n$, denoted by $B^n$.  It is clear that  the equilateral dimension of $B^n$ is equal to that of $\mathds R^n$.  We start by exhibiting some properties of standard equilateral sets in $B^n$.

\begin{proposition}\label{alpha} Let $\{x_1,\dots,x_k\}$ $(k\le n+1)$ be a standard equilateral set in $B^n$.  Then $\Vert c(x_1,\dots,x_k)\Vert\le \alpha_{k+1}$.
\end{proposition}
\begin{proof}
First observe that
\[2\langle x_i,x_j\rangle=\Vert x_i\Vert^2+\Vert x_j\Vert^2-\Vert x_i-x_j\Vert^2\le 1,\]
and therefore
\begin{align*}
\Vert c(x_1,\dots,x_k)\Vert^2\ =\ &k^{-2}\biggl\langle\sum_{\substack{i=1}}^{k}x_i,\sum_{\substack{i=1}}^{k}x_i\biggr\rangle\\
=\ &k^{-2}\biggl [\sum_{\substack{i=1}}^k\Vert x_i\Vert^2+\sum_{\substack{1\le i,\,j\le k\\ i\ne j}}\langle x_i,x_j\rangle\biggr ]\\
\le\ &k^{-2}\biggl [k+\frac{k(k-1)}{2}\biggr]\\
=\ &\alpha_{k+1}^2.
\end{align*}
\end{proof}

In the extremal case $k=n+1$ the bound obtained in Proposition \ref{alpha} can be improved as shown in the next Proposition.  This improvement is needed to prove Proposition \ref{enlargement}.  We first prove a lemma.

\begin{lemma}\label{lemma}Let $\{x_1,x_2,\dots,x_{n+1}\}$ be a maximal standard equilateral set in $\mathds R^n$ with centre at the origin and let $x\in\mathds R^n$ satisfy $\langle x,x_i\rangle\ge 0$ for $i=2,3,\dots,n+1$.  If $\Vert x\Vert \ge 1$, then $\langle x,x_2+x_3+\cdots+x_{n+1}\rangle\ge 1/2$.
\end{lemma}
\begin{proof}
Let $v:=x_2+x_3+\cdots+x_{n+1}$ and let
\[K:=\bigl\{x\in\mathds R^n:\langle x,v\rangle\le 1/2,\,\langle x,x_i\rangle\ge 0\text{ for each }i=2,3,\dots,n+1\bigr\}.\]
$K$ is the intersection of half-spaces and therefore a point of $K$ is an extreme point if and only if it is the intersection of $n$ hyperplanes whose normals form a basis of $\mathds R^n$.  Using the fact that $\langle x_i,x_j\rangle$ is independent of $i,j$ (when $i\neq j$)  it is easy to see that the extreme points of $K$ are $\{0,x_2-x_1,x_3-x_{1},\dots,x_{n+1}-x_{1}\}$.  The norm, being a strictly convex function, i.e.
\[\Vert\lambda x+(1-\lambda)y\Vert<\max(\Vert x\Vert,\Vert y\Vert),\quad x\ne y,\,0<\lambda<1\qquad(\star)\]
takes a maximum value at an extremal point and therefore, since $\Vert x_i-x_{1}\Vert=1$ ($i=2,3,\dots,n+1$), it follows that $\Vert x\Vert \le 1$ for every $x\in K$.  From the strict inequality of $(\star)$ and from the fact that each of the vectors $x_i-x_{1}$ ($i=2,3,\dots, n+1$) lies in the hyperplane $\langle x,v\rangle =1/2$, it follows that if $x\in\mathds R^n$ satisfies $\langle x,x_i\rangle\ge 0$ ($i=2,3,\dots,n+1$) and $\langle x,v\rangle<1/2$, then $\Vert x\Vert<1$.
\end{proof}

\begin{proposition}\label{betaj} Let $\{u_1,\dots,u_{n+1}\}$ be a standard equilateral set in $B^n$.  Then $\Vert c(u_1,\dots,u_{n+1})\Vert\le \beta_{n+1}$.

\end{proposition}
\begin{proof}
Let $\{u_1,u_2,\dots,u_{n+1}\}$ be a maximal standard equilateral set in $B^n$.  Then $\{0,u_2-u_1,\dots,u_{n+1}-u_1\}$ is again a maximal standard equilateral set in $B^n$.  Let us denote its centre by $c$.  Note that $\Vert c\Vert=\beta_{n+1}$.  For each $i=1,2,\dots,n+1$, let $x_i:=u_i-u_1-c$.  Then $\{x_1,x_2,\dots,x_{n+1}\}$ is a maximal standard equilateral set with centre at the origin.  Note that
\[c(u_1,u_2,\dots,u_{n+1})=c(x_1,x_2,\dots,x_{n+1})+u_1+c=u_1+c.\]
Thus
\begin{equation*}
\Vert c(u_1,u_2,\dots,u_{n+1})\Vert ^2=\Vert u_1+c\Vert ^2=\Vert u_1\Vert^2+\Vert c\Vert^2+2\langle u_1,c\rangle,
\end{equation*}
and therefore for the proposition to hold we require
\[\biggl\langle \frac{-u_1}{\Vert u_1\Vert}\,,c\biggr\rangle\ge\frac{\Vert u_1\Vert}{2}.\qquad(\star)\]
To this end we calculate
\begin{align*}
1\ge\Vert u_i\Vert^2=&\Vert x_i+c\Vert^2+\Vert u_1\Vert^2+2\langle u_1,x_i+c\rangle\\
=&1+\Vert u_1\Vert^2+2\langle u_1,x_i\rangle+2\langle u_1,c\rangle
\end{align*}
 which implies
\[\biggl\langle \frac{-u_1}{\Vert u_1\Vert}\,,x_i\bigg\rangle\ge\frac{\Vert u_1\Vert}{2}-\biggl\langle\frac{-u_1}{\Vert u_1\Vert},c\biggr\rangle\qquad(\star\star).\]
for each $i=2,3,\dots,n+1$.  Now, if the right hand side of $(\star\star)$ is $\le 0$, then $(\star)$ is satisfied.  On the other-hand, if the right hand side of $(\star\star)$ is greater than $0$, then Lemma \ref{lemma} can be applied to conclude
\[\frac{\Vert u_1\Vert}{2}\le\frac{1}{2}\le\bigg\langle \frac{-u_1}{\Vert u_1\Vert},x_2+x_3+\cdots+x_{n+1}\bigg\rangle=\bigg\langle \frac{-u_1}{\Vert u_1\Vert},-x_1\bigg\rangle=\bigg\langle \frac{-u_1}{\Vert u_1\Vert},c\bigg\rangle,\]
which completes the proof.
\end{proof}

\begin{proposition}\label{enlargement}
Every standard equilateral set in $B^n$ can be enlarged to one having size $n+1$ such that its members  all lie in $B^n$.
\end{proposition}
\begin{proof}
Let $\{x_1,\dots,x_k\}$ $(1\le k\le n)$ be a standard equilateral set in $B^n$.  We show that there exists a vector $x_{k+1}\in B^n$ such that $\{x_1,\dots,x_k,x_{k+1}\}$ is a standard equilateral set.  The proof will then follow by induction.

Let $N:=\text{{\rm span}}\set{x_i-c(x_1,\dots,x_k)}{1\le i\le k}$ and set $a:=(I-P_N)c(x_1,\dots,x_k)$, where $P_N$ is the projection of $\mathds R^n$ into $N$ and $I$ is the identity.  The intersection of $B^n$ with the translation $a+N$  is a $(k-1)$-dimensional ball with centre $a$ and radius $\sqrt{1-\Vert a\Vert^2}$.  The set $\{x_1,\dots,x_k\}$ is a standard equilateral set in $(a+N)\cap B^n$ and thus, in view of Proposition \ref{betaj}, it follows that $\Vert c(x_1,\dots,x_k)-a\Vert\le \beta_k$.

Set $u:=-\alpha_{k+1}v$, where $v:=a/\Vert a\Vert$ if $a\ne 0$ and any unit vector in $N^\bot$ if $a=0$.  Then $\Vert a+u\Vert\le\Vert u\Vert=\alpha_{k+1}$ since $\alpha_{k+1}\ge\beta_k=\Vert c(x_1,\dots,x_k\}\Vert\ge\Vert a\Vert$.  
Put $x_{k+1}:=c(x_1,\dots,x_k)+u$.  The set $\{x_1,\dots,x_k,x_{k+1}\}$ is a standard equilateral set in $\mathds R^n$.  Moreover,
\begin{align*}
\Vert x_{k+1}\Vert^2&=\Vert c(x_1,\dots,x_k)+u\Vert^2\\
&=\bigl\Vert c(x_1,\dots,x_k)-a\bigr\Vert^2+\Vert a+u\Vert^2\\
&\le \beta_k^2+\alpha_{k+1}^2\\
&=1.
\end{align*}
\end{proof}

\section{Equilateral weights on $B^n$}

In this section we shall prove that the only admissible equilateral weights on the unit ball of $\mathds R^n$ are those that take a constant value.

For any  linear subspace $M$ of $\mathds R^n$, $a\in M$ and $r>0$, we denote the closed  ball in $M$ with centre $a$ and radius $r$ by $B^{M}(a,r)$, i.e. $B^M(a,r)=\set{x\in M}{\Vert x-a\Vert\le r}$.  We will also denote by $S^M(a,r)$ the sphere in $M$ with centre $a$ and radius $r$, i.e. $S^M(a,r)=\set{x\in M}{\Vert x-a\Vert= r}$.  We will write $B(a,r)$ (resp. $S(a,r)$) instead of $B^{\mathds R^n}(a,r)$ (resp. $S^{\mathds R^n}(a,r)$.  We will need the following definition.

\begin{definition}Let $a,b\in B^n$, $a\ne b$  and $N:=(b-a)^\bot$.  For any subspace $M\neq\{0\}$ of $\mathds R^n$ define
\[\gamma^M(a,b):=\sup\biggl\{r> 0\,:\,\frac{a+b}{2}+B^{M\cap N}(0,r)\subseteq B^n\biggr\}.\]
\end{definition}

Note that the set involved in the definition of $\gamma^M(a,b)$ is not empty  and bounded above by $1$.  Instead of $\gamma^{\mathds R^n}(a,b)$ we will simply write $\gamma(a,b)$.  It is easy to see that $\gamma^M(a,b)$ is in fact equal to the maximum of the set of its definition.  In addition, if $M_1$ and $M_2$ are subspaces of $\mathds R^n$ such that $M_1\subseteq M_2$, then $\gamma^{M_2}(a,b)\le\gamma^{M_1}(a,b)$.  The motivation behind this definition lies in the following observation.

\begin{lemma}\label{gamma1}
Let $a,b\in B^n$ such that $\Vert b-a\Vert=2\alpha_{n+1}$ and $\gamma(a,b)\ge \beta_n$.  Then $f(a)=f(b)$ for every equilateral weight $f$ on $B^n$.
\end{lemma}
\begin{proof}
Let $N:=(b-a)^\bot$ and let $\{x_1,\dots,x_n\}$ be a standard equilateral set in
\[\frac{a+b}{2}+S^N(0,\beta_n)\subseteq B^n.\]
Each $x_i$ can be written as $(a+b)/2+n_i$, where $n_i\in N$ and $\Vert n_i\Vert=\beta_n$.  Thus,
\[\Vert x_i-a\Vert^2=\biggl\Vert \frac{b-a}{2}+n_i\biggr\Vert^2=\alpha_{n+1}^2+\beta_n^2=1.\]
Similarly, $\Vert x_i-b\Vert =1$, i.e. $\{a,x_1,\dots,x_{n}\}$ and $\{b,x_1,\dots,x_{n}\}$ are maximal standard equilateral sets in $B^n$, and therefore
\[f(a)+\sum_{i=1}^n f(x_i)=f(b)+\sum_{i=1}^n f(x_i),\]
for every equilateral weight $f$ on $B^n$.
\end{proof}

\begin{lemma}\label{gamma2}
Let $a, b\in B^n$, $a\ne b$ and let $T$ be a two-dimensional subspace of $\mathds R^n$ containing $a$ and $b$.   Then $\gamma^T(a,b)=\gamma(a,b)$.
\end{lemma}
\begin{proof}
We show that $\gamma (a,b)\ge\gamma^T(a,b)$.  Let $u$ be a unit vector in $T$ such that $\langle u,b-a\rangle=0$ and $\langle u,b+a\rangle\ge 0$.  Set $x_0:=(a+b)/2$.  Let $r>0$ such that $\Vert x_0+ru\Vert\le 1$ and let $x\in (b-a)^\bot$ such that $\Vert x\Vert\le r$.  Then $P_Tx=\lambda u$ where $|\lambda|\le \Vert x\Vert\le r$.  Hence
\begin{align*}
\Vert x_0+x\Vert^2=\ &\Vert x_0\Vert^2+\Vert x\Vert^2+2\langle x_0,x\rangle\\
\le\ &\Vert x_0\Vert^2+\Vert x\Vert^2+2|\langle P_T x_0, x\rangle|\\
=\ &\Vert x_0\Vert^2+\Vert x\Vert^2+2|\lambda|\langle x_0,u\rangle\\
\le\ &\Vert x_0\Vert^2+r^2+2r\langle x_0,u\rangle\\
=\ &\Vert x_0+r u\Vert^2\\
\le\ &1,
\end{align*}
and therefore $\gamma (a,b)\ge\gamma^T(a,b)$ as required.
\end{proof}

\begin{lemma}\label{shell}
Let $f$ be an equilateral weight on $B^n$.  There exists $0\le \lambda_n<1$ such that $f$ is constant in $\set{x\in B^n}{\Vert x\Vert\ge\lambda_n}$.
\end{lemma}
\begin{proof}
It suffices to show that there exists $0\le\lambda_n<1$ such that $f$ is constant in $\set{x\in B^n\cap T}{\Vert x\Vert\ge\lambda_n}$ for every two-dimensional subspace $T$ of $\mathds R^n$.

Fix an arbitrary two-dimensional subspace $T$ and let $D$ denote the closed unit disc $B^n\cap T$.  To make calculations easier we fix a rectangular coordinate system in $D$ with origin $o$ at the centre of $D$ (see Figure 1.).  Consider the points $w(0,-1)$, $x(-1,0)$, $y(0,1)$ and $z(1,0)$.  Let $C_w$ (resp. $C_x$, $C_y$, $C_z$)  be the circular arc with centre $w$ (resp. $x$, $y$, $z$) and radius $2\alpha_{n+1}$.  The arcs $C_w$ and $C_x$ meet in $D$ at the point
$a$ the coordinates of which can be easily calculated:
\[a\biggl(\frac{-1+\sqrt{8\alpha_{n+1}^2-1}}{2},\frac{-1+\sqrt{8\alpha_{n+1}^2-1}}{2}\biggr).\]
Similarly, let $b,c,d\in D$ such that $C_x\cap C_y=\{b\}$,  $C_y\cap C_z=\{c\}$ and $C_z\cap C_w=\{d\}$.  Let $C_a$ (resp. $C_b$, $C_c$ and $C_d$) denote the circular arc in $D$ having centre $a$ and radius $2\alpha_{n+1}$ (see Figure 1.).

\begin{figure}
\begin{center}
\begin{pspicture}(-4.5,-4)(4.5,4)
\psset{yunit=30pt}
\psset{xunit=30pt}
\psset{runit=30pt}
\dotnode(0,0){o}
\dotnode(-3.98,0){x}
\dotnode(4,0){z}
\dotnode(0,-4){w}
\dotnode(0,4){y}
\uput[ul](o){$o$}
\uput[l](x){$x$}
\uput[r](z){$z$}
\uput[u](y){$y$}
\uput[d](w){$w$}
\pscircle[linewidth=1pt](o){4}
\psarc[linewidth=1pt](w){6.5}{54.5}{125.5}
\pnode(-1,1.7){Cw}
\uput[ul](Cw){$C_w$}
\psarc[linewidth=1pt](x){6.5}{-35.5}{35.5}
\pnode(1,3.5){Cx}
\uput[dr](Cx){$C_x$}
\pnode(-2.5,-3){D}
\uput[dl](D){$D$}
\dotnode(2.14,2.14){a}
\uput[ur](a){$a$}
\dotnode(-2.14,2.13){d}
\uput[u](d){$d$}
\dotnode(2.17,-2.14){b}
\uput[r](b){$b$}
\dotnode(-2.17,-2.14){c}
\uput[r](c){$c$}
\SpecialCoor
\dotnode(4;-30){g}
\uput[r](g){$g$}
\psarc[linewidth=0.2pt](g){4}{90}{210}
\dotnode(2.7,1.92){h}
\uput[ur](h){$h$}
\psline[linewidth=0.2pt](w)(h)
\dotnode(1.36,-1.04){l}
\uput[u](l){$l$}
\psarc[linewidth=1pt](a){6.5}{199.8}{250.8}
\pnode(-3.3,-1.6){Ca}
\uput[ur](Ca){$C_a$}
\psline[linewidth=0.2pt](w)(g)
\psline[linewidth=0.2pt](g)(h)
\psline[linewidth=0.2pt](l)(g)
\psline[linewidth=0.2pt](w)(a)
\pnode(0,-4.5){f1}
\uput[d](f1){Figure 1.}
\end{pspicture}
\vspace{40pt}
\end{center}
\end{figure}

First we show that $\gamma^T(a,w)\ge\beta_n$.  Let $g$ be the point $\bigl(\frac{\sqrt{3}}{2},-\frac{1}{2}\bigr)$.  Since $2\alpha_{n+1}\le \sqrt{3}$, it easy to see that the circular arc in $D$ having centre $g$ and radius $1$ intersects $C_w$, say at $h$.  Observe that if $l$ is the midpoint of the line segment $wh$, then $|lg|=\beta_n$.  So to show that $\gamma^T(w,a)\ge\beta_n$, it suffices to show that the angle $\widehat{owa}$  is less than or equal to the angle $\widehat{owh}$.  To this end, it is enough to show that $\sin \widehat{owa}\le\sin \widehat{owh}$.  Since $\widehat{doa}=\frac{\pi}{2}$, we have
\begin{align*}
\sin \widehat{owa}=&\sin(\pi/4-\widehat{oaw})\\
=&\frac{1}{\sqrt 2}\bigl(\cos\widehat{oaw}-\sin \widehat{oaw}\bigr).
\end{align*}
Applying the sine rule for triangle $oaw$ we deduce that
\[\sin\widehat{oaw}=\frac{\sin 3\pi/4}{2\alpha_{n+1}}=\frac{1}{2}\sqrt{\frac{n}{n+1}}\quad\text{ and }\quad\cos\widehat{oaw}=\frac{1}{2}\sqrt{\frac{3n+4}{n+1}}.\]
Thus,
\[\sin\widehat{owa}=\frac{1}{2\sqrt 2}\biggl(\sqrt{3+\frac{1}{n+1}}-\sqrt{1-\frac{1}{n+1}}\biggr).\]
On the other-hand
\begin{align*}
\sin\widehat{owh}=&\sin(\pi/3-\widehat{lwg})\\
=&\frac{1}{2}(\sqrt{3}\cos \widehat{lwg}-\sin\widehat{lwg})\\
=&\frac{1}{2}(\sqrt{3}\alpha_{n+1}-\beta_n)\\
=&\frac{1}{2\sqrt 2}\biggl(\sqrt{3+\frac{3}{n}}-\sqrt{1-\frac{1}{n}}\biggr).
\end{align*}
Thus, $\sin\widehat{owa}\le\sin\widehat{owh}$ and therefore $\gamma^T(w,a)\ge \beta_n$.

It is clear (see Figure 1.) that $\gamma^T(u,a)\ge\gamma^T(w,a)$ for every $u\in C_a$.  Thus, in view of Lemma \ref{gamma1} and Lemma \ref{gamma2}, it follows that $f$ is constant on $C_a$.   By symmetry, it follows that $f$ is constant on the circuit $C_a\cup C_b\cup C_c\cup C_d$.  If $\{w',x',y',z'\}$ is another quadruple of points on the circumference of $D$ such that $w'y'$ and $x'z'$ are perpendicular, then we can repeat the same as above to deduce that $f$ is constant on the corresponding circuit joining the points $w'$, $x'$, $y'$ and $z'$.  Moreover, since any two such circuits intersect, it follows that $f$ is constant in the annulus $\set{u\in D}{|ou|\ge 2\alpha_{n+1}-|oa|}$.  Let $\lambda_n:=2\alpha_{n+1}-|oa|$.  From the coordinates of $a$ one can calculate
\[\lambda_n=\frac{1}{\sqrt{2}}\biggl(1+\sqrt{4+\frac{4}{n}}-\sqrt{3+\frac{4}{n}}\,\biggr).\]
\end{proof}

For each $\rho\in [\beta_n,1]$ define $\eta_n(\rho):=\alpha_{n+1}-\sqrt{\rho^2-\beta_n^2}$.  Observe that the value $\eta_n(\rho)$ decreases strictly from $\alpha_{n+1}$ (when $\rho=\beta_n$) to $0$ (when $\rho=1$) and  $\eta_n(\rho)=\rho$ if, and only if, $\rho=\beta_{n+1}$.   Thus, $\eta_n(\rho)\ge \rho$ for every $\rho\in[\beta_n,\beta_{n+1}]$ and $\eta_n(\rho)<\rho$ when $\rho\in(\beta_{n+1},1]$. The geometric meaning of $\eta_n(\rho)$ becomes apparent from the following Lemma.

\begin{lemma}\label{geometric meaning}\begin{enumerate}[(a)]
\item Let $1\ge\rho\ge\beta_n$ and let $x\in B^n$ such that $\Vert x\Vert=\eta_n(\rho)$.  Then there exists a standard equilateral set $\{x_1,x_2,\dots,x_n\}$ such that $\Vert x_i\Vert=\rho$ and $\Vert x_i-x\Vert =1$ for every $i=1,2,\dots,n$.

\item Conversely, if $\{x_1,x_2,\dots,x_{n+1}\}$ is a maximal standard equilateral set  
in $B^n$ and $\Vert x_i\Vert=\rho$ for every $i=1,2,\dots,n$, then $\rho\ge \beta_n$ and if $\hull{x_1,\dots,x_{n+1}}$ contains $0$, then $\Vert x_{n+1}\Vert=\eta_n(\rho)$.
\end{enumerate}
\end{lemma}
\begin{proof} {\rm(a)}~First note that if $\rho=1$, then $0=\eta_n(\rho)=\Vert x\Vert$ and therefore the statement is true in this case.  Suppose that $\beta_n\le\rho<1$. Let $\{u_1,u_2,\dots,u_n\}$ be a maximal standard equilateral set in $x^\bot$ with centre $0$.  Then $\Vert u_i\Vert=\beta_n$.  It is easy to check that the vectors
\[x_i:=u_i-\sqrt{\rho^2-\beta_n^2}\,\frac{x}{\Vert x\Vert}\qquad\qquad (i=1,2,\dots,n)\]
 satisfy the required conditions.

 {\rm(b)}~The locus of points in $\mathds R^n$ equidistant from each of the $x_i$'s $(i=1,\dots,n)$ is the line passing through $0$ and parallel to $x_{n+1}-c(x_1,\dots,x_n)$.  The point on this line with shortest distance  from any (and therefore from each) of the $x_i$'s $(i=1,\dots,n)$ is that with position vector $c(x_1,\dots,x_n)$.  Thus
\[\beta_n=\bigl\Vert c(x_1,\dots,x_n)-x_i\bigr\Vert\le\Vert x_i\Vert=\rho\qquad(i=1,2,\dots,n).\]

If $0\in\hull{x_1,\dots,x_{n+1}}$, then $0=\lambda x_{n+1}+(1-\lambda)c(x_1,\dots,x_n)$ for some $\lambda \in[0,1]$.  Thus
\begin{align*}
\alpha_{n+1}=\bigl\Vert x_{n+1}-c(x_1,\dots, x_n)\bigr\Vert&=\Vert x_{n+1}\Vert+\bigl\Vert c(x_1,\dots, x_n)\bigr\Vert\\
&=\Vert x_{n+1}\Vert+\sqrt{\rho^2-\beta_n^2}.
\end{align*}
\end{proof}

\begin{lemma}\label{constant}Let $f$ be an equilateral weight on $B^n$ taking the constant value $\delta$ in $\set{x\in B^n}{\Vert x\Vert\ge \rho_0}$, where $\rho_0\in[\beta_n,1]$.  Then $f$ takes the constant value $W-n\delta$ in $ B(0,\eta_n(\rho_0))$ where $W$ is the weight of $f$.
If $\rho_0\le\beta_{n+1}$, then $f$ takes the constant value $\frac{W}{n+1}$ in $B^n$.
\end{lemma}
\begin{proof}
Let $x\in  B(0,\eta_n(\rho_0))$.  The inequality $0\le \Vert x\Vert\le\eta_n(\rho_0)$ implies that there exists $1\ge\rho\ge\rho_0$ such that $\eta_n(\rho)=\Vert x\Vert$.  Thus, by Lemma \ref{geometric meaning}, there are vectors $\{x_1,x_2,\dots,x_n\}$ such that $\Vert x_i\Vert =\rho$ for $1\le i\le n$ and such that $\{x,x_1,x_2,\dots,x_n\}$ is a maximal standard equilateral set in $B^n$.  So, $f(x)+n\delta=W$.

If $\rho_0\le\beta_{n+1}$, then $\eta_n(\rho_0)\ge \rho_0$, i.e.
\[\set{x\in B^n}{\Vert x\Vert\ge \rho_0}\cap  B(0,\eta_n(\rho_0)) \ne\emptyset, \]
and thus $W-n\delta=\delta$.
\end{proof}

We are now ready to prove the result announced in the abstract.

\begin{theorem}Every equilateral weight on $B^n$ is constant.
\end{theorem}
\begin{proof}
Set $\mu_n(\rho):=1-\eta_n(\rho)$ and $\nu_n(\rho):=\rho-\mu_n(\rho)$ when $\rho\in[\beta_n,1]$.   Observe that $\mu_n$ is strictly increasing with range $[1-\alpha_{n+1}, 1]$.  It is easy to check that $\nu_n$ is strictly decreasing and that $\nu_n(1)=0$.  Thus, $\mu_n(\rho)<\rho$ for all $\rho\in[\beta_n,1)$.

Let $f$ be an equilateral weight on $B^n$.  In view of Lemma \ref{shell} we can define
\[ \theta:=\inf\set{\rho}{f \text{ is constant in } B^n\setminus B(0,\rho)}\]
and note that $\theta\le \lambda_n$.  In view of Lemma \ref{constant}, the proof would be complete if we could show that $\theta<\beta_{n+1}$.  So we suppose that $\theta\ge\beta_{n+1}$ and seek a contradiction.  Let $\epsilon$ be a positive real number  satisfying
\[\epsilon<\min\{\nu_n(\lambda_n),\beta_{n+1}-\beta_n\}.\]
Then $\theta-\epsilon>\beta_n>1-\alpha_{n+1}$ and thus $\mu_n^{-1}(\theta-\epsilon)$ is defined.  In addition, it follows that $\mu_n^{-1}(\theta-\epsilon)>\theta$, for if $\mu_n^{-1}(\theta-\epsilon)\le\theta$, then (since $\mu_n$ is strictly increasing)  we would have $\theta-\epsilon\le \mu_n(\theta)$ and this would lead to  $\epsilon\ge\nu_n(\theta)\ge\nu_n(\lambda_n)$, which contradicts our choice of $\epsilon$.

Fix $\rho_0:=\mu_n^{-1}(\theta-\epsilon)$.  Then, since $\mu_n^{-1}(\theta-\epsilon) >\theta$, $f$ takes a constant value, say $\delta$, in the annulus $\{x\in B^n:\Vert x\Vert\ge \rho_0\}$ and therefore, by virtue of Lemma \ref{constant},  $f$ takes the constant value $W-n\delta$ in  $ B(0,\eta_n(\rho_0))$, where $W$ is the weight of $f$.  We  show that $f$ then must take the constant value $\delta$ in the annulus  $\{x\in B^n:\Vert x\Vert\ge \mu(\rho_0)\}$.  This would contradict the definition of $\theta$ and thus conclude the proof.

To this end, fix and arbitrary vector $u\in B^n$ such that
\[1-\eta_n(\rho_0)=\mu_n(\rho_0)\le\Vert u\Vert\le\rho_0,\qquad\qquad(\star)\]
and let $v=-\frac{1-\Vert u\Vert}{\Vert u\Vert }u$.  Then $v\in B^n$ and $1=\Vert u-v\Vert=\Vert u\Vert+\Vert v\Vert$.  From the inequalities
\[1-\eta_n(\rho_0)+\Vert v\Vert\le \Vert u\Vert +\Vert v\Vert=1\le\rho_0+\Vert v\Vert\]
we obtain $1-\rho_0\le\Vert v\Vert\le \eta_n(\rho_0)$ and therefore, in virtue of Lemma \ref{constant}, we obtain $f(v)=W-n\delta$.  We can now apply Proposition \ref{enlargement} to obtain an enlargement $\{x_1,\dots,x_{n-1},u,v\}$ of $\{u,v\}$ to a maximal standard equilateral set in $B^n$.  Let $w:=(u+v)/2$. For each $i=1,2,\dots,n-1$ we have
\[\Vert x_i\Vert^2=\Vert x_i-w\Vert^2+\Vert w\Vert^2=\frac{3}{4}+\biggl|\Vert u\Vert-\frac{1}{2}\biggr|^2.\]
If $\eta_n(\rho_0)>\frac{1}{2}$, then $\rho_0^2<5/4-\alpha_{n+1}$ and thus
\[\Vert x_i\Vert^2\ge\frac{3}{4}>\frac{5}{4}-\frac{1}{\sqrt{2}}>\frac{5}{4}-\alpha_{n+1}>\rho_0^2 .\]
On the other-hand, if $\eta_n(\rho_0)\le\frac{1}{2}$, then $(\star)$ implies
\[\frac{1}{2}\le 1-\eta_n(\rho_0)\le\Vert u\Vert\]
and therefore
\begin{align*}
\Vert x_i\Vert^2&=\frac{3}{4}+\biggl |\Vert u\Vert-\frac{1}{2}\biggr |^2\\
&\ge\frac{3}{4}+\biggl(\frac{1}{2}-\eta_n(\rho_0)\biggr)^2\\
&=1-\eta_n(\rho_0)+\eta_n(\rho_0)^2\\
&=(1-2\alpha_{n+1})\biggl(\sqrt{\rho_0^2-\beta_n^2}-\alpha_{n+1}\biggr)+\rho_0^2\\
&\ge\rho_0^2.
\end{align*}
So in both cases we conclude that $f(x_i)=\delta$ for each $i=1,2,\dots,n-1$ and therefore
\begin{align*}
f(u)=&W-f(v)-\sum_{i=1}^{n-1}f(x_i)\\
=&W-(W-n\delta)-(n-1)\delta=\delta,
\end{align*}
as required.  This completes the proof.
\end{proof}

\begin{remark}
\begin{enumerate}[{\rm(i)}]
\item It follows immediately from the theorem proved here that an equilateral weight on a connected subset of $\mathds R^n$ that is the union of unit balls, is constant.
\item Our method of the proof should work also to show that an equilateral weight on an $n$-dimensional (closed) ball with radius greater than $\alpha_{n+1}$ is constant.  What is not completely clear to us is the case when the radius lies in the interval $(\beta_{n+1},\alpha_{n+1}]$.
\item Although we have defined equilateral weights as real-valued functions, it is apparent from the proof that the same conclusion can be drawn if one considers group-valued equilateral weights on the unit ball of $\mathds R^n$.  
\end{enumerate}
\end{remark}

\end{document}